\documentclass[a4paper,12pt]{amsart}

\usepackage{float}
\usepackage{euscript,eufrak,verbatim}
\usepackage{graphicx}
\usepackage[usenames]{color}
\usepackage[colorlinks,linkcolor=red,anchorcolor=blue,citecolor=blue]{hyperref}
\usepackage{amsmath}
\usepackage{amsthm}
\usepackage[all]{xy}
\usepackage{graphicx} 
\usepackage{amssymb} 
\usepackage{bbm}

\newtheorem{thm}{Theorem}[section]
\newtheorem{prop}[thm]{Proposition}

\newtheorem{lem}[thm]{Lemma}
\newtheorem{cor}[thm]{Corollary}

\newtheorem{conj}[thm]{Conjecture}

\newtheorem{prob}[thm]{Problem}

\def\N{\mathbb{N}}
\def\Z{\mathbb{Z}}
\def\N{\mathbb{N}}
\def\R{\mathbb{R}}
\def\C{\mathbb{C}}

\def\FF{\mathcal{F}}

\def\FFi{\overline{\mathcal{F}}}
\def\SS{\mathcal{S}}
\def\wdim{\text{\rm Widim}}
\def\mdim{\text{\rm mdim}}

\def\Int{\text{\rm Int }}
\def\dist{\text{\rm dist}}
\def\diam{\text{\rm diam}}

\makeatletter
\@namedef{subjclassname@2020}{%
	\textup{2020} Mathematics Subject Classification}

\makeatletter 
\@addtoreset{equation}{section}
\numberwithin{equation}{section}

\newcommand{\norm}[1]{\left\lVert#1\right\rVert}

\title[]{Lowering mean topological dimension}

\author{Ruxi Shi}

\address{Shanghai Center for Mathematical Sciences, Fudan University, 200438 Shanghai, China}

\email{ruxishi@fudan.edu.cn}

\thanks{The author was partially supported by the New Cornerstone Science Foundation through the New Cornerstone Investigator Program and NSFC No. 12231013. }

\keywords{}
\subjclass[2020]{}

\thanks{} 

\begin{document}
	
	\maketitle

\begin{abstract}
In this paper, we prove that for a topological dynamical system with positive mean topological dimension and marker property, it has factors of arbitrary small mean topological dimension and zero relative mean topological dimension which separate points.
\end{abstract}
\section{Introduction}

A topological dynamical system is a pair $(X, T)$ where $X$ is compact metrizable space and $T : X \to X$ is a homeomorphism. The fundamental invariant of dynamical systems is topological entropy, which has been the subject of extensive study for many years. Mean dimension is a new invariant of topological dynamical systems introduced by Gromov \cite{G}. Heuristically, it counts the number of real-valued parameters per unit time needed for describing a system. Moreover, the mean dimension of a dynamical system vanishes when it has finite topological entropy. From this perspective, the mean dimension measures the complexity of dynamical systems of infinite entropy.

For a dynamical system $(X,T)$, it is said to be {\it aperiodic} if $T^nx=x$ implies $n=0$. Moreover, it is said to satisfy the {\it marker property} if for any positive integer $N$ there exists an open set $U \subset X$ satisfying that 
$$
U \cap T^n U = \emptyset \text{ for } 0 < n < N \text{ and }  X=\cup_{n\in \Z} T^n U.
$$
For example, a dynamical system with an aperiodic minimal factor has the marker property \cite[Lemma 3.3]{L99}. It is clear that the marker property implies the aperiodicity. But the inverse is proved to be false recently \cite{tsukamoto2022g, shi2021marker}.

The marker property has been extensively investigated in the theory of mean dimension, for example, the dynamical embedding problem.  Gutman and Tsukamoto \cite{gutman2020embedding} showed that a dynamical system having the marker property with mean dimension smaller than $N/2$ can be embedded in the shift of Hilbert cube $([0,1]^N)^\Z$. Lindenstrauss and Tsukamoto \cite{lindenstrauss2019double} proved that if a dynamical system has the marker property, then there exists a metric such that the upper metric mean  dimension is equal to mean dimension of this dynamical system. The {\it metric mean  dimension} was introduced by Lindenstrauss and Weiss \cite{LindenstraussWeiss2000MeanTopologicalDimension} which  majors the value of mean topological dimension. See also \cite{Gut15Jaworski,GutLinTsu15, gutman2017embedding,tsukamoto2020potential} and references therein where the marker property was investigated. Lindenstrauss \cite{L99} proved that for a dynamical system with marker property and zero mean dimension has factors of arbitrary small topological entropy which separate points. More precisely, he proved that:

\begin{thm}[\cite{L99}, Corollary 6.9]\label{L thm}
If $(X,T)$ is a dynamical system with  marker property and $\mdim(X,T)=0$, then for any $\delta>0$ and any two distinct points $z\not=z'\in X$, there exists a factor $(Y,S)$ of $(X,T)$ via $\pi$ such that 
$$
\pi(z)\not=\pi(z') \text{ and } h_{top}{(Y, S)}<\delta.
$$
\end{thm} 

The result cited, \cite[Corollary 6.9]{L99}, gives the theorem assumed there that $(X,T)$ has an aperiodic minimal factor, what is actually used in the proof is that $(X,T)$ has the marker property. 

A natural question is asked whether the similar phenomenon is for dynamical systems with positive mean dimension, i.e. how can we lower the complexity of factors of dynamical system with positive mean dimension? However, we cannot expect the same conclusion of Theorem \ref{L thm} for dynamical system with positive mean dimension, because if a dynamical system with positive mean dimension has factors of finite topological entropy which separate points, then it will be an inverse limit of dynamical system with zero mean dimension and thus have zero mean dimension itself \cite[Proposition 6.11]{L99}.  Therefore, it turns out to study how small mean dimension of factors of a dynamical system with positive mean dimension are. Meanwhile, we keep the condition that the mean dimension of factor map vanishes (it is natural in Theorem \ref{L thm} because the given dynamical system is of zero mean dimension which implies that the mean dimension of fibres always vanishes). Our main result is stated as follow.
\begin{thm}\label{main thm}
Let $(X,T)$ be a dynamical system of positive mean dimension. Suppose $(X,T)$ has the marker property. Then for any $\delta>0$ and any $z\not=z'\in X$, there exists a factor $(Y,S)$ of $(X,T)$ via $\pi$ such that 
$$
\pi(z)\not=\pi(z'), \mdim{(Y,S)}<\delta \text{ and } \mdim{(\pi, T)}=0.
$$
\end{thm} 

As we mentioned before, the condition $\mdim{(\pi, T)}=0$ holds automatically in Theorem \ref{L thm} because $\mdim{(\pi, T)}\le \mdim(X,T)=0$. On the other hand, the mean dimension of $ \mdim{(Y,S)}$ is necessarily positive by \cite[Theorem 1.5]{tsukamoto2022analogue}\footnote{It states that if $\pi : (X, T) \to (Y, S)$ is a factor map between dynamical systems and $\mdim(Y, S) = 0$ then
$\mdim(X, T) = \mdim(\pi, T)$.}. As a corollary of Theorem \ref{main thm}, we could present any dynamical system with positive mean dimension and marker property as a product of dynamical systems with arbitrary small mean dimension. 

\begin{cor}\label{main cor}
    Let $(X,T)$ be a dynamical system of positive mean dimension. Suppose $(X,T)$ has the marker property. Then for any $\delta>0$, there exists a family $\{(Y_n, S_n) \}_{n\in \N}$ of dynamical systems with $\mdim(Y_n, S_n)<n\delta$ such that 
    $$
    (X, T)= \varprojlim_{n\in \N} (Y_n, S_n).
    $$
\end{cor}

\begin{proof}
    Let $\Delta(X)=\{ (x,x): x\in X \} \subset X\times X$. For any factor $\pi_Y: (X, T)\to (Y, S)$, we denote by 
    $$
    N(\pi_Y):=\{(x,x'): \pi_Y(x)\not=\pi_Y(x') \}\subset (X\times X) \setminus \Delta(X),
    $$
    which is an open set. By Theorem \ref{main thm}, the family consisting of $N(\pi_Y)$ with $\mdim(Y, S)<\delta$ forms a cover of $(X\times X) \setminus \Delta(X)$. Since $(X\times X) \setminus \Delta(X)$ is a Lindel\"of space, there is a countable subfamily of $\{N(\pi_{Z_n})\}_{n\in \N}$ which is again a cover of $(X\times X) \setminus \Delta(X)$. Let $\pi_N:=\prod_{n=0}^{N} \pi_{Z_n}$ and $Y_N:=\pi_N(X)$.
    Then the factor map $\pi:=\varprojlim_{n\in \N} \pi_{n}: X\to \varprojlim_{n\in \N} Y_n$ is a $\Z$-equivalent continuous surjection which separates points and is consequently injective. We conclude that $(X, T)= \varprojlim_{n\in \N} (Y_n, S_n).$
\end{proof}

On the other hand, Theorem \ref{main thm} is related to a problem of Hurewicz theorem for mean dimension. 
\begin{prob}\label{prob}
Let $\pi : (X,T) \to (Y,S)$ be a factor map. Does the following inequality hold true?
\begin{equation}\label{eq:prob}
    \mdim{(X,T)}\le \mdim{(Y,S)}+\mdim{(\pi, T)}.
\end{equation}
\end{prob}

This problem was originally proposed by Tsukamoto in \cite[Problem 4.8]{Tsu08} as an analogue of Hurewicz theorem \cite[P91, Theorem VI 7]{hurewicz1942dimension} which states that for a continuous map $f: X\to Y$ between compact metrizable spaces, one has that 
$$
\dim{(X)}\le \dim{(Y)}+\sup_{y\in Y}\dim{f^{-1}(y)}.
$$
Recently, Tsukamoto gave a negative answer to Problem \ref{prob} by constructing a counterexample \cite[Theorem 1.3]{tsukamoto2022analogue}. From this point of view, Theorem \ref{main thm} tells that for any dynamical system with marker property and positive mean dimension we can always find a factor such that the equation \eqref{eq:prob} doesn't not hold.

Finally, we remark that Theorem \ref{main thm} also works for $\Z^d$-action with $d>1$. Since the approach is similar but more complicated, we keep the proof of $\Z$-action in current paper.

\subsection*{Organization of the paper}
Firstly we recall the definition and basic properties of mean dimension in Section \ref{sec:Preliminary}. In Section \ref{sec:dynamical tiling} we discuss the dynamical tiling for the dynamical system with marker property. In Section \ref{sec:Width dimension of fibres}, we investigate the width dimension of fibres. In Section \ref{sec:lowering}, we present the proof of Theorem \ref{main thm}. Finally in Section \ref{sec:remark}, we discuss some open problems. 

\subsection*{Acknowledgement}

I would like to express my gratitude to Tsukamoto Masaki. This work initially arose from discussions with him. However, he graciously declined co-authorship.

\section{Preliminary}\label{sec:Preliminary}

In this section, we recall some basic notions in dynamical system, in particular, the theory of mean dimension.
\subsection{Mean dimension}
Let $d$ be a compatible distance on a compact metrizable space $X$. 
	 For a set $Z$ and $\epsilon>0$, a map $f:X\to Z$ is called \textit{$(d, \epsilon)$-injective} if $\diam(f^{-1}(z))<\epsilon$ for all $z\in Z$ where $\diam$ denotes the diameter under the metric $d$.  Gromov \cite{G} introduced the {\it $\epsilon$-width}, denoted by $\wdim_{\epsilon}(X, d)$,
which is  the
smallest integer $n$ such that there exists an $( \epsilon, \epsilon$-injective continuous map $f : X \rightarrow P$
from $X$ into some $n$-dimensional simplicial complex $P$. 
Gromov showed that the unit ball $B$ of a $n$-dimensional Banach space $(E,\|\cdot\|)$  satisfies
$$\forall 0<\epsilon<1, \ \wdim_{\epsilon}(B,\|\cdot\|)=n.$$

	Let  $T:X\to X$ be a homeomorphism. The {\it mean (topological) dimension} of $(X, T)$ is defined by
	$$
	\mdim(X, T)=\lim_{\epsilon \rightarrow 0} \mdim_\epsilon(X, T, d),
	$$
	where we let $$\mdim_\epsilon(X, T, d)=\lim_{n\to \infty} \frac{\wdim_{\epsilon}(X, d_n)}{n} \text{ and } d_n(\cdot, \cdot)=\max_{0\le i\le n-1} d(T^i \cdot, T^i \cdot).$$ The existence of the limit  follows from  the sub-additivity of the sequence $\left\{\wdim_{\epsilon}(X, d_n)\}\right)_{n\ge 1}$. In general, for a closed subset $A\subset X$, we define the {\it upper and lower mean dimensions} of $A$ by
	$$
	\overline{\mdim}(A, T)=\lim_{\epsilon \rightarrow 0} \overline{\mdim}_\epsilon(A, T, d),
	$$
	and 
	$$
	\underline{\mdim}(A, T)=\lim_{\epsilon \rightarrow 0} \underline{\mdim}_\epsilon(A, T, d),
	$$
	where $\overline{\mdim}_\epsilon(A, T, d)=\limsup_{n\to \infty} \frac{\wdim_{\epsilon}(A, d_n)}{n}$ and $\underline{\mdim}_\epsilon(A, T, d)=\liminf_{n\to \infty} \frac{\wdim_{\epsilon}(A, d_n)}{n}$.

	We mention some basic properties of mean dimension that we use in the current paper. We refer to the book \cite{Coo05} for the proofs and further properties. 
	\begin{itemize}
		\item If $A\subset B$ are two closed subsets of $X$, then \begin{equation*}\label{sub}\overline{\mdim}(A,T)\le \overline{\mdim}(B,T).\end{equation*}
		\item For $n\in \N$, \begin{equation*}\label{power}\mdim(X,T^n)=n\cdot \mdim(X,T).\end{equation*}
		\item For dynamical systems $(X_i,T_i)$, $1\le i\le n$, we have
		$$
		\mdim(\prod_{i=1}^n X_i, \prod_{i=1}^n T_i)\le \sum_{i=1}^n \mdim(X_i, T_i).
		$$
	\end{itemize}	
	
\subsection{Relative mean dimension}
Let $\pi :(X, T) \to (Y, S)$ be a factor map between two dynamical systems, i.e.  $\pi: X \to Y$ is a continuous surjection satisfying $\pi\circ T= S\circ \pi$. The {\it relative mean dimension} of $\pi$ is defined by
$$
\mdim(\pi, T)=\lim_{\epsilon\to 0}\lim_{n\to \infty} \frac{\sup_{y\in Y}\wdim_\epsilon(\pi^{-1}(y), d_n)}{n}.
$$
The limit (with respect to $n$) exists because of the sub-additivity of $\sup_{y\in Y}\wdim_\epsilon(\pi^{-1}(y), d_n)$. Tsukamoto \cite[Proposition 1.1]{tsukamoto2022analogue} proved the following proposition which clarifies the meaning of this definition.
\begin{prop}
    Let $\pi :(X, T) \to (Y, S)$ be a factor map between two dynamical systems. Then 
    $$
\mdim(\pi, T)=\sup_{y\in Y} \overline{\mdim}(\pi^{-1}(y),T)=\sup_{y\in Y} \underline{\mdim}(\pi^{-1}(y),T).
    $$
\end{prop}
It tells us that the relative mean dimension $\mdim(\pi, T)$ properly measures the mean dimension of fibers of $\pi$.




\section{Dynamical tiling}\label{sec:dynamical tiling}

A collection $\{W_n\}_{n\in \Z}$ of Borel sets is said to be a {\it tiling} of $\R$  if $\cup_{n\in \Z}W_n=\R$ and the Lebesgue measure of $W_n\cap W_m$ vanishes for distinct $n,m\in \Z$. Moreover, it is called a {\it convex tiling} of $\R$ if $W_n$ are convex for all $n\in \Z$, i.e. $W_n$ are intervals on $\R$.

A {\it dynamical tiling} of a dynamical system $(X,T)$ is a family of convex tilings $\mathcal{W}=\{\mathcal{W}_x\}_{x\in X}$ where $\mathcal{W}_x=(W(x,n))_{n\in \Z}$ satisfies that 
\begin{itemize}
    \item the tiling $\{W(x,n) \}_{n\in \Z}$ is $\Z$-equivariant, i.e. $$W(T^nx,m)=-n+W(x, n+m);$$
    \item the tiles $W(x, n)$ depend continuously on $x\in X$ in the Hausdorff topology.
\end{itemize}
For $r>0$ we set $\partial_r [a,b]=[a-r, a+r] \cup [b-r, b+r]$. For a tiling $W=\{W_n\}_{n\in \Z}$, we define
$$
\partial_r W = \cup_{n\in \N} \partial_r W_n.
$$
For a Borel subset $E$ of $\R$, we let the {\it density} of $E$ be
$$
D(E)=\limsup_{R\to \infty} \frac{\left|[-R,R] \cap E \right| }{2R},
$$
where $|\cdot|$ stands for the Lebesgue measure on $\R$. 

In this section, we prove that for a dynamical system with marker property, it has a dynamical tiling with properties as follows.

\begin{prop}\label{prop:tiling 2}
Let $(X,T)$ be a dynamical system. Suppose $(X,T)$ has the marker property. Let $r>0$ and $z\not=z'\in X$. For any $\delta>0$, there exists $K>0$, a $\Z$-equivalent continuous map $\Phi: (X,T) \to ([0,2]^\Z, \sigma)$ and a dynamical tiling $\mathcal{W}=\{\mathcal{W}_x\}_{x\in X}$ satisfying that 
\begin{enumerate}
    \item $\mdim(\Phi(X), \sigma)<\delta$;
    \item $D(\partial_r \mathcal{W}_x )<\delta$ for all $x\in X$;
    \item For each $x\in X$, there exists $n$ such that $W(x,n) \setminus \partial_r W(x,n)\not=\emptyset$ and $W(x,n)\subset [-K, K]$;
    \item $\Phi(z)\not=\Phi(z')$.
\end{enumerate}
\end{prop}

In the rest of this section, we are going to prove Proposition \ref{prop:tiling 2}. Let $(X,T)$ be a topological dynamical system with marker property. Let $r>0$ and $z\not=z'\in X$. Let $M>0$ be a large integer which will be precised later. Then by marker property there exists an integer $M_1>M$, an open set $U$ and a compact set $F$ such that 
\begin{itemize}
    \item $z\in F\subset U\subset X$ and $z'\notin U$;
    \item $
U\cap T^{-n}U=\emptyset~\text{for}~0<|n|<M~\text{and}~X=\cup_{n\in \Z} T^{n}U;
$
\item $X=\cup_{|n|<M_1} T^nF$.
\end{itemize}
Now choose a continuous map $\phi: X\to [0,1]$ satisfying that supp$(\phi)\footnote{Here the support $\phi$  means that supp$(\phi):=\overline{\{x\in X: \phi(x)\not=0 \}}$. }\subset U$ and $\phi=1$ on $F$. By the choice of $\phi$, we have $\phi(z)=1$ and $\phi(z')=0$.  Let
$$
\mathcal{S}(x)=\left\{\left(n, \frac{1}{\phi(T^nx)}\right): n\in \Z, \phi(T^nx)>0  \right\}
$$
be the discrete set in $\R^2$. The {\it Voronoi tiling} $\mathcal{V}(x)=\{V(x,n)\}_{n\in \Z}$ is defined as follows: if $\phi(T^nx)=0$ then $V(x,n)=\emptyset$; if $\phi(T^nx)>0$ then 
$$
V(x,n)=\left\{u\in \R^2: \left|u-\left(n, \frac{1}{\phi(T^nx)}\right)\right|\le \left|u-p\right|, \forall p\in \mathcal{S}(x)\right\}.
$$
It is clearly that for each $x$, the sets $\{V(x,n)\}_{n\in \Z}$ form a tiling of $\R^2$. Let $\pi: \R^2\to \R$ be the projection on the first coordinate. Let $H=(M_1+1)^2$ and
$$
W(x,n)=\pi \left(V(x,n)\cap (\R\times \{-H \}) \right).
$$
Then the sets $\{W(x,n)\}_{n\in \Z}$ form a tiling of $\R$. Moreover, the tiling $\{W(x,n) \}_{n\in \Z}$ is $\Z$-equivariant, i.e. $W(T^nx,m)=-n+W(x, n+m)$.

By \cite[Lemma 4.1]{GutLinTsu15} (see also \cite[Lemma 6.2]{gutman2020embedding} and \cite[Claim 5.11]{gutman2019application}), we have the following basic properties of tiling $\{W(x,n) \}_{n\in \Z}$. 

\begin{lem}\label{lem:gut}
Let $x\in X$ and $n\in \Z$ with $\phi(T^nx)>0$. Then the following properties hold.	
\begin{itemize}
	\item [(1)] $V(x,n)$ contains the ball $B_{M/2}(n, 1/\phi(T^nx))$.
	\item [(2)] $W(x,n)$ is contained in $[n-M_1-1, n+M_1+1]$.
	\item [(3)] If $W(x,n)\not=\emptyset$ then $\phi(T^nx)>1/2$.
\end{itemize}
\end{lem}

Let $R:=\max\{r, 9 \}$.  We pick a number $c$ satisfying that
$$
1<c<\frac{1}{1-\delta}.
$$
We choose $M$ so large that
$$
M>\max\left\{ \frac{2R(c+1)}{c-1}, \frac{2}{(1-\delta)^{-1}-c}  \right\}.
$$

By using Lemma \ref{lem:gut}, following the same lines of \cite[Lemma 4.2]{GutLinTsu15} (also \cite[Lemma 6.2]{gutman2020embedding} and \cite[Lemma 2.7]{gutman2019application}), we have the following lemma. 
\begin{lem}\label{lem0}
We have the following properties.
\begin{itemize}
    \item[(1)] For all $n\in \Z$, $|W(x,n)\setminus \partial_R W(x,n)|\ge (1-\delta)|W_c(x,n)|$ where $W_c(x,n)=\pi \left(V(x,n)\cap (\R\times \{-cH \}) \right)$.
    \item[(2)] $D(\partial_R \mathcal{W}_x )<\delta$ for all $x\in X$.
\end{itemize}
   
\end{lem}
Since the proofs of Lemma \ref{lem:gut} and Lemma \ref{lem0} are a largely reproduction of the ones in \cite[Lemma 4.1, Lemma 4.2]{GutLinTsu15}, we leave them to the readers to work it out the details. Now we are going to show that the dynamical tiling $\mathcal{W}$ satisfies the property in Proposition \ref{prop:tiling 2} (3).

\begin{lem}
    For each $x\in X$, there exists $n$ such that $W(x,n) \setminus \partial_r W(x,n)\not=\emptyset$ and $W(x,n)\subset [-2M_1-2, 2M_1+2]$.
\end{lem}
\begin{proof}
    Let $n$ be the integer such that $0\in W_c(x,n)$ with $|W_c(x,n)|>0$. Following the same proof of Lemma \ref{lem:gut} (2), we have 
    $$
    W_c(x,n)\subset [n-M_1-1, n+M_1+1].
    $$
    It follows that $|n|\le M_1+1$. Then by Lemma \ref{lem:gut} (2) and Lemma \ref{lem0} (1), we obtain 
    $$
    W(x,n) \setminus \partial_r W(x,n)\not=\emptyset,
    $$
    and 
    $$
    W(x,n) \subset [n-M_1-1, n+M_1+1] \subset  [-2M_1-2, 2M_1+2].
    $$
\end{proof}

Now we are going to define the function $\Phi$. Define $\gamma(x)=\frac{2}{1+e^{-|x|}}$ which has a unique maximal value at $0$ (see Figure \ref{fig1}). 
\begin{figure}[H] 
\centering 
\includegraphics[width=0.8\linewidth]{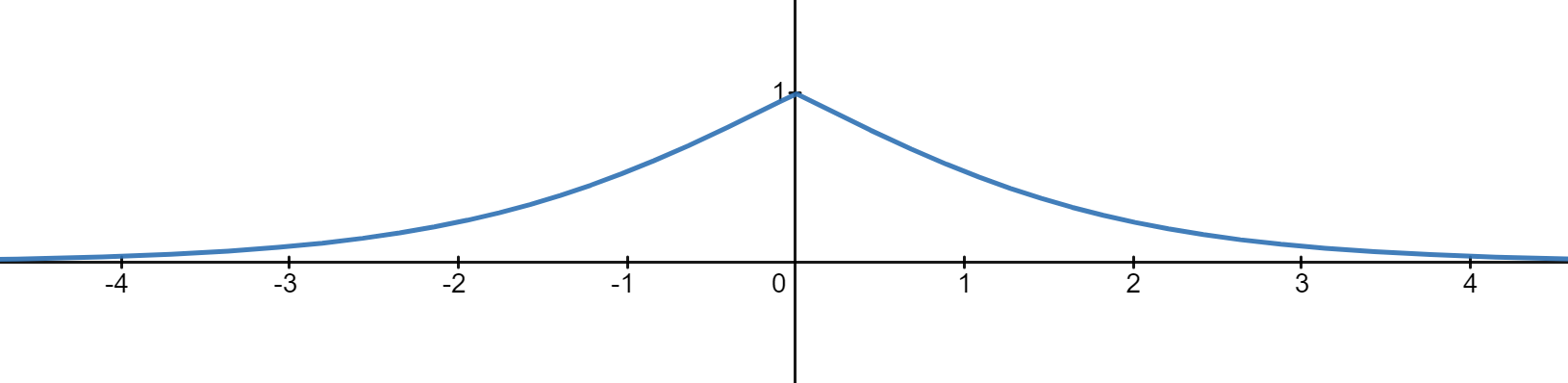}
\caption{The graph of $\gamma(x)$.} 
\label{fig1} 
\end{figure}
Let $\alpha: [0, +\infty) \to [0,1]$ be a continuous function  such that 
    \begin{itemize}
       \item $\alpha(t)=0$ for $0\le t\le 2$;
        \item $\alpha(t)=1$ for $t\ge R/3$.
    \end{itemize}
Now we define $h: X\to [0,2]$ to be a continuous function as follows: if $0\in \partial \mathcal{W}_x$ then $h(x)=0$; otherwise, there is a unique $n$ such that $0\in \Int W(x,n)$, and let 
$$
h(x)=\min\{ \dist(0, \partial \mathcal{W}_x ), 1\}+\alpha(\dist(0, \partial \mathcal{W}_x ))\gamma(n),
$$
where dist is the Euclidean distance.
Define
$\Phi: X\to [0,2]^\Z$ by $x\mapsto (h(T^nx))_{n\in \Z}$.
It is clear that $\Phi$ is continuous and $\Z$-equivariant. We will show the function $\Phi$ is what we desire in Proposition \ref{prop:tiling 2}. Notice that the image of $\Phi(x)$ is related to the tiling $\mathcal{W}_x=\{W(x,n)\}_{n\in \Z}$: 
\begin{lem}\label{lem:ob}
Let $x\in X$ and $n\in \Z$. Let $\Phi(x)=(t_k)_{k\in \Z}$. Then we have $t_k\le 1+\gamma(n-k)$ for all $k\in  \Z\cap W(x,n)$. The equality holds if and only if $\dist(k, \partial \mathcal{W}_x )\ge R/3$.
\end{lem}
\begin{proof}
For $k\in  \Z\cap \Int W(x,n)$, we have $0\in \Int W(T^kx,n-k)$ and consequently
$$
t_k=\min\{ \dist(k, \partial \mathcal{W}_x ), 1\}+\alpha(\dist(k, \partial \mathcal{W}_x ))\gamma(n-k)\le 1+\gamma(n-k).
$$
Notice that the equality holds if and only if 
$\alpha(\dist(k, \partial \mathcal{W}_x ))=1 $ which is equivalent to $\dist(k, \partial \mathcal{W}_x )\ge R/3$.
 
\end{proof}

Recall that $z\not=z'$ are two points which are fixed before. We have the following lemma which justifies Proposition \ref{prop:tiling 2} (4).
\begin{lem}\label{lem1}
We have $\Phi(z)\not=\Phi(z')$.
\end{lem}
\begin{proof}
We claim that $\Phi(z')_0<\Phi(z)_0=1+\gamma(0)$. This implies directly that $\Phi(z)\not=\Phi(z')$. It is sufficient to show our claim.

Since $z\in F$ and consequently $\phi(z)=1$, by definition of $V(z,n)$ we have $(l,-H) \in V(z,0)$ for $|l|<M$ and therefore $[-M, M]\subset W(z,0)\not=\emptyset$. It follows that $0\in W(z,0)\setminus \partial_{R/3} W(z,0)$. Then by Lemma \ref{lem:ob} we have $\Phi(z)_0=1+\gamma(0)$.

    Since $z'\notin U$ and consequently $\phi(z')=0$, by definition of $V(z',n)$ we have $W(z',0)=\emptyset$. Therefore, by Lemma \ref{lem:ob} we have $\Phi(z')_0\le 1+\gamma(n)<1+\gamma(0)$ for some $n\not=0$ with $0\in W(x,n)$. We conclude that $\Phi(z)\not=\Phi(z')$.
\end{proof}

Now we are going to show Proposition \ref{prop:tiling 2} (1).
\begin{lem}\label{lem3}
    $\mdim(\Phi(X), \sigma)<\delta$.
\end{lem}
\begin{proof}
     Let $P_N: [0,2]^\Z \to [0,2]^N$ be the projection to the $0, 1, 2, \cdots, (N-1)$-th coordinates. By Lemma \ref{lem0} and Lemma \ref{lem:ob}, the set $P_N(\Phi(X))$ is contained in 
\begin{equation}
    \begin{split}
        E_N:=\{ e\in [0,1]^N: ~& e_k =1+\gamma(n_i-k) \text{ for } a_i\le k\le b_i, \\
        & 0\le a_0 <b_0<a_1<b_1<\dots < a_L<b_L\le N, \\
        & n_1< n_2 <\dots <n_L,\\
        &
        \frac{N-\sum_{i=0}^{L} (a_{i}-b_i)}{N}\le 
        \delta N \},
    \end{split}
\end{equation}
for $N$ large enough. Since $\dim(E_N)\le \delta N$, we conclude that
$$
\mdim(\Phi(X),\sigma)\le \limsup_{N\to \infty} \frac{P_N(\Phi(X))}{N} \le \delta.
$$
\end{proof}

Now we are at the position to show our main result in this section.
\begin{proof}[Proof of Proposition \ref{prop:tiling 2}]
    Combining Lemma \ref{lem0}, lemma \ref{lem:ob}, Lemma \ref{lem1} and Lemma \ref{lem3}, we conclude the result.
\end{proof}

\section{Width dimension of fibres}\label{sec:Width dimension of fibres}

Gromov proved the following statement in \cite[p.107]{gromov1988width} which shows that a direct analogue of the Hurewicz theorem for $\epsilon$-width dimension does not hold.
\begin{lem}
Let $(X, d)$ be a $(2n + 1)$-dimensional compact Riemannian manifold. For any $\epsilon>0$ there exists a smooth map $f : X \to [0, 1]$ such that for every $t \in [0, 1]$ we have
$$
\wdim_\epsilon(f^{-1}(t), d)\le n.
$$
\end{lem}
Tsukamoto proved the following variation of Gromov’s lemma \cite[Corollary 3.5]{tsukamoto2022analogue}.
\begin{lem}\label{lem:masaki}
Let $K$ be a simplicial complex with a metric $d$. For any $\epsilon>0$ and any natural number $m$ there exists a continuous map $F : K \to [0, 1]^{m-1}$ such that for every point $p \in [0, 1]^{m-1}$ we have
$$
\wdim_\epsilon(f^{-1}(p), d)\le \frac{\dim K}{m}.
$$
\end{lem}

In this section, we prove another variation of Gromov’s lemma involving dynamical systems. The following lemma is a crucial ingredient of the proof of Theorem \ref{main thm}.
\begin{lem}\label{lem:key}
Let $(X,T)$ be a topological dynamical system with metric $d$. For any $\epsilon>0$ and any natural numbers $n,m$ there exists a continuous map $F : X \to [0, 1]^{m-1}$ such that for every point $p \in [0, 1]^{m-1}$ we have
$$
\wdim_\epsilon(F^{-1}(p), d_n)\le \frac{\wdim_{\epsilon/2}(X, d_n)}{m}.
$$
\end{lem}
\begin{proof}
By definition of $\epsilon/2$-width dimension, there is an $(\epsilon/2, d_n)$-embedding $f_n: X \to K_n$ where $K_n$ is a simplicial complex with topological dimension $\wdim_{\epsilon/2}(X, d_n)$. Pick a metric $D$ on $K_n$. Then there is an $\delta>0$ such that $D(f_n(x),f_n(y))<\delta$ implies $d_n(x,y)<\epsilon$. By Lemma \ref{lem:masaki}, there exists a continuous function $G_n:K_n \to [0,1]^{m-1}$ such that 
for every point $p \in [0, 1]^{m-1}$ we have
$$
\wdim_\delta(G_n^{-1}(p), D)\le \frac{\wdim_{\epsilon/2}(X, d_n)}{m}.
$$
This means that for each $p \in [0, 1]^{m-1}$ there is an $(\delta, D)$-embedding $g_n$ from $ G_n^{-1}(p)$ to a simplicial complex $E_{n,p}$ with $\dim(E_{n,p})\le \frac{\wdim_{\epsilon/2}(X, d_n)}{m}.$ It follows that $$g_n \circ f_n: f_n^{-1}(G_n^{-1}(p)) \to E_{n,p} $$
is an $(\epsilon, d_n)$-embedding. Thus we have 
$$
\wdim_\epsilon(f_n^{-1}(G_n^{-1}(p)), d_n)\le \dim(E_{n,p})\le \frac{\wdim_{\epsilon/2}(X, d_n)}{m}.
$$
Then we conclude that the continuous funcion $F=G_n \circ f_n: X\to [0, 1]^{m-1}$ is desired.
\end{proof}

\section{Lowering mean dimension of systems with marker property}\label{sec:lowering}

In this section, we prove our main result in current paper. The following proposition is the key point towards Theorem \ref{main thm}.

\begin{prop}\label{main prop}
Let $(X,T)$ be a dynamical system with marker property. Let $z\not=z'\in X$. For any $\epsilon, \delta>0$ there exist a dynamical system $(Y, S)$ and a factor map $\pi: (X,T) \to (Y,S)$ such that 
$$
\pi(z)\not=\pi(z'), ~\mdim(Y,S)<\delta $$ and $$\lim_{n\to \infty} \frac{\sup_{y\in Y}\wdim_\epsilon(\pi^{-1}(y), d_n)}{n}<\delta.
$$
\end{prop}

\begin{proof}[Proof of Theorem \ref{main thm} by assuming Proposition \ref{main prop}]
     Apply Proposition \ref{main prop} to $\epsilon=\frac{1}{n}$ and $\frac{\delta}{2^n}$ for each positive integer $n$. It follows that for each $n$ there is a factor map $\pi_n: (X,T) \to (Y_n,S_n)$ such that $\pi_n(z)\not=\pi_n(z')$,
    $$
\mdim(Y_n,S_n)<\frac{\delta}{2^n} \text{ and } \lim_{N\to \infty} \frac{\sup_{y\in Y}\wdim_{\frac{1}{n}}(\pi^{-1}(y_n), d_N)}{N}<\frac{\delta}{2^n}.
$$
    Let 
    $$
    \pi:=\prod_{n\ge 1} \pi_n : X \to \prod_{n\ge 1} Y_n, S=\prod_{n\ge 1} S_n \text{ and } Y=\pi (X).
    $$
    Obviously we have $\pi(z)\not=\pi(z')$.
    It is sufficient to show that the factor map $\pi: (X,T) \to (Y,S)$ is what we desire. Firstly, we have 
    $$
    \mdim(Y,S) \le \sum_{n\ge 1} \mdim(Y_n,S_n)<\sum_{n\ge 1}\frac{\delta}{2^n}= \delta.
    $$ 
    Next, since
    $$
    \pi^{-1}(y)=\cap_{n\ge 1} \pi_n^{-1}(y_n) \text{ for } y=(y_n)_{n\ge 1}\in Y,
    $$
     we have
     \begin{align*}
         &\lim_{N\to \infty} \frac{\sup_{y\in Y}\wdim_{\frac{1}{n}}(\pi^{-1}(y), d_N)}{N} \\
         \le & \lim_{N\to \infty} \frac{\sup_{y\in Y}\wdim_{\frac{1}{n}}(\pi^{-1}(y_n), d_N)}{N}<\frac{\delta}{2^n},
     \end{align*}
     for each $n\ge 1$. Finally, we conclude that $$\mdim(\pi, T)=\lim_{n\to \infty}\lim_{N\to \infty} \frac{\sup_{y\in Y}\wdim_{\frac{1}{n}}(\pi^{-1}(y), d_N)}{N}=0.$$
     
\end{proof}

In the rest of this section, we prove Proposition \ref{main prop}. 
\begin{proof}[Proof of Proposition \ref{main prop}]
Let $n$ be a positive integer such that 
$$
\frac{\wdim_{\epsilon/2}(X, d_n)}{n}< \mdim_{\epsilon/2}(X, T, d)+1
$$
Let $m$ be a positive integer such that 
$$m>\frac{\mdim_{\epsilon/2}(X, T, d)+1}{\delta}.$$

  Let $\delta'< \frac{\mdim_{\epsilon/2}(X, T, d)+1}{2m}$ which is smaller than $\delta/2$.  By Proposition \ref{prop:tiling 2}, there exists $K>0$, a $\Z$-equivalent continuous map $\Phi: (X,T) \to ([0,2]^\Z, \sigma)$  and a dynamical tiling $\mathcal{W}$ satisfying that 
\begin{itemize}
\item $\mdim(\Phi(X), \sigma)<\delta'$;
    \item $D(\partial_{3m} \mathcal{W}_x )<\delta'$ for all $x\in X$;
    \item For each $x\in X$, there exists $l$ such that $W(x,l) \setminus \partial_{3m} W(x,l)\not=\emptyset$ and $W(x,l)\subset [-K, K]$;
    \item $\Phi(z)\not=\Phi(z')$.
\end{itemize}
Denote the image $\Phi(X)$ by $Z$. Then $\mdim(Z, \sigma)< \delta'<\delta/2$.  

By Lemma \ref{lem:key}, there exists a continuous map $F : X \to [0, 1]^{m-1}$ such that for every point $p \in [0, 1]^{m-1}$ we have
$$
\wdim_\epsilon(F^{-1}(p), d_n)\le \frac{\wdim_{\epsilon/2}(X, d_n)}{m}.
$$

Let $\alpha: [0, +\infty) \to [0,1]$ be a continuous map  such that 
    \begin{itemize}
        \item $\alpha(0)=0$;
        \item $\alpha(t)=0$ for $t\ge 3m$;
        \item $\alpha(t)=1$ for $1\le t\le 2m$.
    \end{itemize}
Now we define a continuous function $g: X\to [0,1]$ as follows. Take $x\in X$. If $0\in \partial W_x$, then $g(x)=0$, where  $\partial W_x= \cup_{n\in \Z} \partial W(x,n)$. Otherwise, there is a unique $b\in \Z$ such that $0\in \Int  W(x,b)$. Choosing $a\in \Z$ satisfying $a\equiv b \mod m-1$ and $0\in [a, a+m-2]$, we set 
$$
g(x)=\alpha(\dist(0, \partial W_x)) F(T^a(x))_{-a}.
$$
where we identify $[0,1]^{m-1}=[0,1]^{\{0,1,\dots, m-2\}}$.   Since $\mathcal{W}_x$ depends continuously on $x$, the function $g$ is continuous. Let 
$$
I_g: X \to [0,1]^\Z, x\mapsto (g(T^t(x)))_{t\in \Z},
$$
be a continuous map. 
Denote $I_g(X)$ by $Y$. Then $I_g$ is a factor map from $(X,T)$ to $(Y, \sigma)$. Let 
$$
\pi: (X,T) \to (Y,\sigma)\times (Z, \sigma), x\mapsto (I_g(x), \Phi(x)).
$$
We will show $\pi$ is what we desire.

For each $x\in X$,  since there exists $l$ such that $W(x,l) \setminus \partial_{3m} W(x,l)\not=\emptyset$ and $W(x,l)\subset [-K, K]$, we can find $a\equiv n \mod m-1$ such that $[a,a+m-2]\subset \Int W(x,l)$ and $$1\le \min \{ \dist(a, \partial W_x), \dist(a+m-2, \partial W_x) \}\le 2m.$$ It follows that $|a|\le K$.
By Lemma \ref{lem:11}, we have 
$$
(g(T^t(x)))_{a\le t\le a+m-1}=F(T^a(x)),
$$
for all $x\in X$ and $a$ which depends on $x$ as above.
It follows that 
$$
\pi^{-1}(y, z)\subset \cup_{a\in [-K,K]} F^{-1}(y_a, \dots, y_{a+m-2}) 
$$
and consequently
$$
\wdim_\epsilon(\pi^{-1}(y, z), d_{n})\le \frac{\wdim_{\epsilon/2}(X, d_n)}{m},
$$
for all $(y, z)\in Y\times Z$. Then we have 
\begin{align*}
    &\lim_{N\to \infty} \frac{\sup_{(y, z)\in Y\times Z}\wdim_\epsilon(\pi^{-1}(y, z), d_N)}{N}\\
    =&\inf_{N\in \N} \frac{\sup_{(y, z)\in Y\times Z}\wdim_\epsilon(\pi^{-1}(y, z), d_N)}{N}\\
    \le &\frac{\sup_{(y, z)\in Y\times Z}\wdim_\epsilon(\pi^{-1}(y, z), d_{n})}{{n}}\\
    \le &\frac{\wdim_{\epsilon/2}(X, d_n)}{nm}\\
    \le &\frac{\mdim_{\epsilon/2}(X, T, d)+1}{m}
    <\delta.
\end{align*}
Since $\mdim(Z,\sigma)<\delta/2$, it remains to show that $\mdim(Y,\sigma)\le \delta/2$. Let $P_N: [0,1]^\Z \to [0,1]^N$ the projection to the $0, 1, 2, \cdots, (N-1)$-th coordinates. By Lemma \ref{lem:12}, $P_N(Y)$ is contained in 
$$
E_N=\{ e\in [0,1]^N: e_n =0 \text{ except for at most } \delta'N+1 \text{ entries} \},
$$
for $N$ large enough. Since $\dim(E_N)\le \delta'N+1$, we conclude that
$$
\mdim(Y,\sigma)\le \limsup_{N\to \infty} \frac{P_N(Y)}{N} \le \delta'<\delta/2.
$$
\end{proof}

It remains to prove the following lemmas which is used in the proof of Proposition \ref{main prop}.
\begin{lem}\label{lem:11}
Let $a,b\in \Z$ with $a\equiv b \mod m-1$. If $[a,a+m-2]\subset \Int W(x,b)$ and $1\le \min \{ \dist(a, \partial W_x), \dist(a+m-2, \partial W_x) \}\le 2m$, then 
$$
(g(T^t(x)))_{a\le t\le a+m-2}=F(T^a(x)).
$$
\end{lem}
\begin{proof}
Take $a\le t\le a+m-2$. Then $W(T^tx,n-t)=-t+W(x,n)$ and consequently 
$$
1\le \dist(0, \partial W_{T^tx})\le 2m.
$$
It follows that 
$$
g(T^tx)=F(T^{a-t}T^{t}x)_{-a+t}=F(T^ax)_{-a+t}.
$$
\end{proof}

\begin{lem}\label{lem:12}
If a integer $a\notin \partial_{3m} W_x$, then $g(T^a(x))=0$.
\end{lem}
\begin{proof}
Since $W(T^ax,n-a)=-a+W(x,n)$ and consequently 
$$
\dist(0, \partial W_{T^ax})> 3m.
$$
We conclude that $g(T^ax)=0$.
\end{proof}

\section{Remarks and problems}\label{sec:remark}

It is conjectured that one can drop the marker property in Theorem \ref{L thm}. Similarly, we propose the following conjecture.
\begin{conj}\label{conj:1}
Let $(X,T)$ be a dynamical system. Let $z\not=z'\in X$. Then for any $\delta>0$ there exists a factor $(Y,S)$ of $(X,T)$ via $\pi$ such that 
$$
\pi(z)\not=\pi(z'), \mdim{(Y,S)}<\delta \text{ and } \mdim{(\pi, T)}=0.
$$
\end{conj}

As we mentioned before, Theorem \ref{main thm} tells that for any dynamical system with marker property and positive mean dimension we can find a factor such that it doesn't satisfy \eqref{eq:prob}. We then conjecture that the assumption of marker property can be dropped off. Clearly, Conjecture \ref{conj:1} implies Conjecture \ref{conj:2}.

\begin{conj}\label{conj:2}
Let $(X,T)$ be a dynamical system of finite non-zero mean dimension. Then there exists a factor $(Y,S)$ of $(X,T)$ via $\pi$ such that 
$$
\mdim{(X,T)}> \mdim{(Y,S)}+\mdim{(\pi, T)}.
$$
\end{conj}


\bibliographystyle{alpha}
\bibliography{universal_bib}

\end{document}